\documentclass[reqno]{amsart}
\usepackage{amsmath,amssymb,amsfonts}
\usepackage{amssymb,amscd,color}
\usepackage{pdfsync}
\usepackage{graphics}
\usepackage{graphicx,color}
\usepackage{url}

\theoremstyle{plain}
\newtheorem{prob}{Problem}

\newtheorem{theorem}{Theorem}[section]

\newtheorem{lemma}[theorem]{Lemma}

\newtheorem{corollary}[theorem]{Corollary}

\theoremstyle{definition}
\newtheorem{definition}[theorem]{Definition}

\def\va{{\bf a}}   
   
  \def\vk{{\bf k}} 
   
  \def\vs{{\bf s}} 
\def\vu{{\bf u}}  \def\vw{{\bf w}} \def\vx{{\bf x}} 
\def\vy{{\bf y}} \def\vz{{\bf z}}

\def\vA{{\bf A}} \def\vB{{\bf B}} \def\vC{{\bf C}} \def\vD{{\bf D}}
 \def\vF{{\bf F}}  
\def\vI{{\bf I}}   
\def\vM{{\bf M}}   
  \def\vS{{\bf S}} \def\vT{{\bf T}}
\def\vU{{\bf U}}    
 \def\vZ{{\bf Z}}

\def\CC{{\mathbb C}}

\def\RR{{\mathbb R}}
\def\ZZ{{\mathbb Z}}

\def\Prob{{\mathbb P}}

\def\Exp{{\mathbb E}}

\def\supp{\operatorname{supp}}

\def\dopp{{f}}
\def\aR{{\va_R(\beta)}}
\def\aRp{{\va_R(\beta')}}

\def\aT{{\va_T(\beta)}}
\def\aTp{{\va_T(\beta')}}

\def\Std{{\vS_{\tau,\dopp}}}
\def\Stdp{{\vS_{\tau',\dopp'}}}

\def\op{{\text{op}}}

\def\vAt{\tilde{\vA}}
\def\vxt{{\vz}}

\begin{document}
\title[Random sensor arrays and Kerdock codes]{Accurate detection of 
moving targets via random sensor arrays and Kerdock codes}
\author[Thomas Strohmer]{Thomas~Strohmer}
\address[Thomas Strohmer]{Department of Mathematics\\
University of California\\ Davis, CA 95616}
\email
{strohmer@math.ucdavis.edu}

\author[Haichao Wang]{Haichao~Wang}
\address[Haichao Wang]{Department of Mathematics\\
University of California\\ Davis, CA 95616}
\email 
{hchwang@ucdavis.edu}

\keywords{Sparsity, Radar, Compressive Sensing, Random Sensor Arrays, MIMO, Kerdock Codes}
\maketitle

\begin{abstract}
The detection and parameter estimation of moving targets is one of the
most important tasks in radar. Arrays of randomly distributed antennas
have been popular for this purpose for about half a century. Yet,
surprisingly little rigorous mathematical theory exists for random arrays
that addresses fundamental question such as how many targets can be 
recovered, at what resolution, at which noise level, and with which algorithm. 
In a different line of research in radar, mathematicians and engineers
have invested significant effort into the design of radar transmission
waveforms which satisfy various desirable properties. In this paper
we bring these two seemingly unrelated areas together. Using tools from 
compressive sensing we derive a theoretical framework for the recovery of 
targets in the azimuth-range-Doppler domain via random antennas arrays. 
In one manifestation of our theory we use Kerdock codes as transmission 
waveforms and exploit some of their peculiar properties in our analysis.
Our paper provides two main contributions: (i) We derive the first 
rigorous mathematical theory for the detection of moving targets using random 
sensor arrays. (ii) The transmitted waveforms satisfy a variety of properties
that are very desirable and important from a practical viewpoint. Thus our
approach does not just lead to useful theoretical insights, but is also
of practical importance. Various extensions of our results are derived
and numerical simulations confirming our theory are presented.
\end{abstract}

\section{introduction}
\label{s:intro}


The detection and parameter estimation of moving targets is one of the
most important radar applications. The use of antenna arrays greatly 
improves our ability to perform this task. Antenna arrays make it possible
to estimate not only the range and Doppler frequency, but also 
the azimuth of the target. Furthermore, using multiple antennas
can significantly increase signal strength and thus in turn can greatly
enhance accuracy and our ability to locate low contrast targets
(``faint'' or ``weak'' targets).

Therefore it does not come as a surprise that in recent years radar systems
employing multiple antennas at the transmitter and the receiver (also referred 
to as MIMO radar, where MIMO stands for multiple-input multiple-output)
have attracted enormous attention in the engineering and signal processing
community. Despite the significant resources that have been devoted to
MIMO radar, there exists fairly little rigorous mathematical theory 
for MIMO radar that addresses fundamental questions, such as how many
targets can be detected at which azimuth-range-Doppler resolution and at
what signal-to-noise ratio. Existing theory focuses mainly on the detection of
a single target~\cite{mimobook,LiStoica2007}. Only very recently, in the 
footsteps of compressive sensing, do we see the emergence of a rigorous
mathematical theory for MIMO radar that addresses the more realistic and
more interesting case of multiple targets~\cite{SF12}.
However, for the widely popular case of randomly spaced antennas\footnote{In
this paper we only consider the case of co-located transmitters and receivers, 
which is the most relevant situation in practice. We do not discuss the
case of widely separated antennas~\cite{HaimovichBlumCimini2008}.},
the mathematical theory is still in its infancy. 

In an independent and seemingly disparate line of research in radar, mathematicians and engineers
have devoted substantial efforts to the design of radar transmission waveforms
that satisfy a variety of desirable properties. 
The vast majority of this research has focused on single antenna radar systems,
and it is a priori not clear if and how these waveforms can be utilized for
MIMO radar. In this paper we bring together these two independent areas of
research, MIMO radar with random antenna arrays and radar waveform design,
by developing a rigorous mathematical framework for accurate target detection via
random arrays, which at the same time utilizes some of the most attractive 
radar waveforms, such as Kerdock codes.

%

\medskip

A radar system illuminates a region of interest in order to detect the 
location, velocity, and reflectivity of the objects (targets) in its field 
of view. We consider the following standard (narrowband) radar 
model~\cite{Rihk}. Suppose a target located at \emph{range} $r$ is traveling 
with \emph{constant velocity} $v$ and has \emph{reflection coefficient}
$a$. Suppose further just for the moment that we have only one target,
one transmitter and one receiver (in which case we cannot detect direction).
After transmitting signal $s(t)$, the receiver observes the reflected
signal
\begin{equation} \label{eq:Radar_reflected_signal}
y(t) \;=\; a\:\! s(t-\tau_r)\:\! e^{2\pi i\omega_v t}
\end{equation}
where $\tau_r = 2r/c$ is the round trip time of flight,
$c$ is the speed of light, $\omega_v \approx -2\omega_0 v/c$ is the
Doppler shift, and $\omega_0$ is the carrier frequency. The basic 
idea is that the \emph{range-velocity} information $(r,v)$ of the
target can be inferred from the observed \emph{time delay-Doppler
shift} $(\tau_r,\omega_v)$ of $s$ in (\ref{eq:Radar_reflected_signal}).
For only one target this can be done conveniently by correlating the 
received signal $y$ with time-frequency shifted versions of the 
transmitted signal. Since we are dealing with bandlimited signals, it 
suffices to consider discrete signals sampled at a properly chosen
rate $\Delta_t$. It is therefore common practice to compute
\begin{equation}
\label{stft}
V(\tau,\omega):= \sum_l  y(l\Delta_t) \overline{\vs(l\Delta_t-\tau) e^{2\pi i\omega l}}
\end{equation}
and then locate the largest value of $|V(\tau,\omega)|$ in order to detect
the target in the range-Doppler domain.

In the presence of multiple targets more sophisticated methods are
necessary. In order to resolve azimuth in addition
to range and Doppler, we need to employ an array of antennas.
We assume an array of $N_T$ transmit and $N_R$ receiver antennas that
are co-located (also known as mono-static radar)
as illustrated in Figure~\ref{fig:mimoradar}.
A more detailed description of the setup is postponed to
Section~\ref{s:setup}.
\begin{figure}
\begin{center}
\includegraphics[width=80mm]{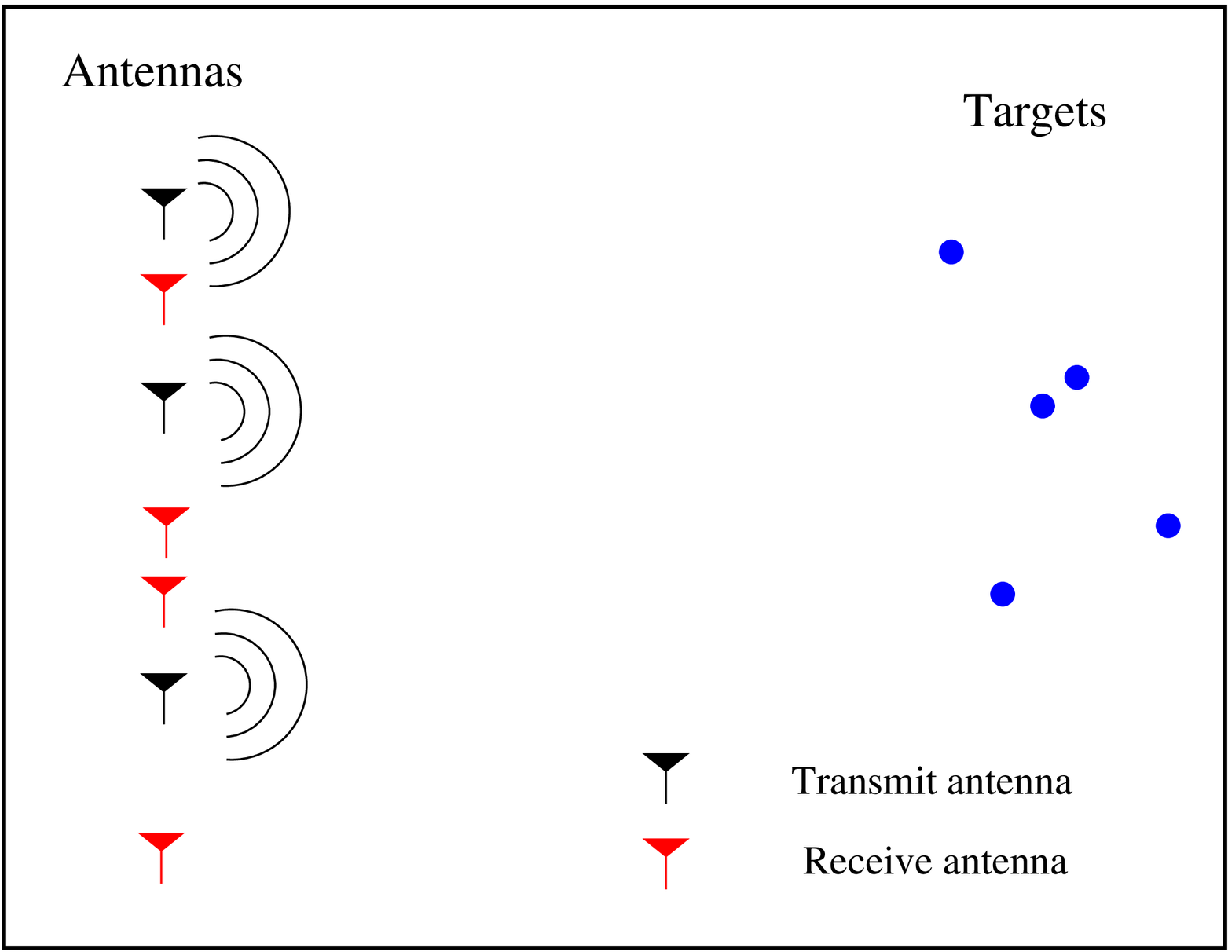}
\label{fig:mimoradar}
\caption{}
\end{center}
\end{figure}
The transmit antennas send simultaneously probing signals, which can
differ from antenna to antenna and can be chosen to our specifications.
It is convenient to divide the region of interest into 
range-azimuth-Doppler cells corresponding to distance, direction and 
velocity, respectively. Let $\vA$ be a measurement matrix whose columns 
correspond to the signal recorded at each receive antenna from a 
single unit-strength scatterer at
a specific range-azimuth-Doppler cell. Let $\vx$ denote a vector whose
elements represent the complex amplitudes of the scatterers. In many cases
the radar scene is sparse in the sense that only a small fraction (often
a {\sl very} small fraction) of the cells is occupied by the objects of
interest. In this case most of the entries of $\vx$ will be zero,
but we do not know which ones, otherwise we would have located
the targets already. With $\vw$ representing a noise vector, we
are faced with the linear system of equations 
\begin{equation}
\label{inverseproblem}
\vy = \vA \vx +\vw,
\end{equation}
where $\vy$ is a vector of measurements collected by the receive antennas 
over an observation interval. Typically this system will be 
{\em underdetermined}, which implies that it will have infinitely many 
solutions. What comes to our rescue here is the sparsity of $\vx$.
While conventional radar processing techniques do not take full advantage 
of sparsity of the radar scene, the recent development of 
compressive sensing provides us with the possibility to optimally utilize 
this property~\cite{HS09, PEPC10, SF12}. 
The approach pursued in this paper to obtain a sparse solution
of~\eqref{inverseproblem} is based on the lasso~\cite{T96}, which gained 
tremendous popularity in connection with compressive sensing. The lasso
solves
\begin{equation}
\label{Lasso2}
\underset{\vx}{\min}\quad \frac{1}{2}\|\vA\vx - \vy\|_2^2 + \lambda \|\vx\|_1,
\end{equation}
where the parameter $\lambda>0$ trades off goodness of fit with sparsity.

However, one of the main challenges in bringing compressive sensing theory
into radar is that in radar the sensing matrix $\vA$ cannot be freely chosen.
Its structure is dictated by the laws of physics on which radar is based. 
The crux is to carefully balance the desired resolution in the 
azimuth-range-Doppler domain with the degrees of freedom at our disposal in 
the formation of $\vA$, such as the antenna locations and the transmit 
waveforms.

A great deal of work has been devoted in the mathematical and engineering
literature to the design of radar transmission waveforms, see for 
instance~\cite{Alltop,BBW12,Cos84,GG05,GHS08,PCM08,Wel60} for a small sample of references.
The design criteria for radar waveforms can be roughly split into two categories:
(i) properties that are important from the viewpoint of hardware
implementation, and (ii) properties that are relevant
for target detection. Waveforms that fall in the first category are for
example polyphase sequences\footnote{A polyphase sequence is a sequence
whose coefficients are of the form $e^{2\pi i t_k/p}$ for some
$t_k \in \{0,\dots,p-1\}$, see e.g.~\cite{GG05}.},
since they give rise to signals with
low peak-to-average power ratio\footnote{The peak-to-aver power ratio of a
signal $\vs$ is defined, up to different normalizations, as 
$\frac{\|\vs\|_{\infty}}{\|\vs\|_2}$.} (PAPR). 
A low PAPR is desirable in the digital-to-analog conversion of signals, since
signals with large PAPR would require expensive power amplifiers.
Polyphase sequences also have the advantage that they can be very
convientiently implemented in hardware via simple look-up tables.
The second category usually includes waveforms with 
low auto-correlation and (nearly) ideal ambiguity function. Quite a number of polyphase
sequences, such as Alltop sequences or Kerdock codes, fall in this category.
With the exception of~\cite{HS09} a rigorous mathematical theory 
concerning the benefits (or even optimality) of such sequences has only existed 
for the detection of a single target. 
Common to all these carefully constructed sequences in both categories is that they have been
designed for single-antenna radar systems and it is a priori not clear at all if any
of these sequences are useful in exploiting the potential benefits of a MIMO radar system. 

Our paper provides two main contributions: (i) We derive the first rigorous
mathematical theory for the detection of moving targets in the
azimuth-range-Doppler domain for random sensor arrays.
(ii) The transmitted waveforms satisfy a variety of properties
that are very desirable and important from a practical viewpoint. 
In particular, we show that Kerdock sequences, which would perform very 
poorly in single-antenna radar, are nearly ideally suited for MIMO radar
with randomly spaced antennas. Since Kerdock codes are polyphase sequences, 
they have excellent PAPR and they are easy to implement in hardware via
a simple look-up table. Thus, our framework does not just lead to 
useful theoretical insights, but also has a very strong practical appeal. 

\subsection{Connections with prior work and innovations}
\label{ss:priorwork}

Random sensor arrays have been around for half a century. The pioneering
work~\cite{L64b,L64a} by Lo contains a mathematical analysis of
important specific characteristics of random arrays, such as sidelope
behavior and antenna gain. There is extensive engineering literature
that deals with random arrays in connection with phased array radar
technology, e.g.\ see~\cite{FTD00}. Recently, Carin made an explicit 
connection between the areas of random sensor arrays and compressive 
sensing~\cite{Car09}. He has shown that algorithms developed in
these two seemingly different areas are in fact highly inter-related. The setup in \cite{Car09} is quite different from ours, since the paper is only 
concerned with angular resolution (thus transmission waveforms do not even
explicitly enter into the model), while it is often crucial in practice 
to be able to estimate range and Doppler as well. Moreover, the theoretical
analysis in~\cite{Car09} follows more an engineering style and places less
emphasis on mathematical rigor. The paper~\cite{CBBGS09} provides 
interesting results for the angular estimation of stationary targets. 
Its setup is similar to that in~\cite{Car09}, and quite different from
ours, as it does not deal with waveform design nor with moving targets.

Kerdock codes have been proposed for radar in~\cite{HCM06}.
However in the setting of a single transmit antenna.
Kerdock codes are known to perform rather poorly\footnote{This poor
performance is caused by Property~(ii) in Theorem~\ref{th:Kerdock}.} even 
in the case of single targets as considered in~\cite{HCM06}.
Only in the setting of mulitple transmit antennas can Kerdock codes
exhibit their enormous potential. Our paper utilizes some properties
of Kerdock codes proved in~\cite{HCM06}, but otherwise there is no
overlap. In our paper we also 
present an extension of the main result, that allows for instance
the use of the so-called finite harmonic oscillator system as transmission
waveforms. These sequences have been derived in~\cite{GHS08}, where
the authors also briefly sketch their use in a single-antenna radar
system for the simple case of a single target. Thus, while our framework 
allows one to employ the finite harmonic oscillator system, there is essentially
no overlap of our results with those in~\cite{GHS08}.

The paper \cite{SF12} (coauthored by one of the authors) is closest to this
paper, but the setting is in a sense complementary. \cite{SF12} considers 
a MIMO radar setting with a very specific (non-random) choice for the 
antenna locations, but random waveforms, while the current paper deals
with randomly spaced antennas, but very specific, deterministic waveforms.
At first glance, the difference may appear to be mainly semantic. But in
practice, the second setting has many advantages. From an engineer's
viewpoint random waveforms have several drawback over properly designed
deterministic waveforms: they are much
harder to implement on a digital device (requiring more complicated
hardware, more memory, ...); and they exhibit a larger peak-to-average-power
ratio. On the other hand it makes no difference from
the viewpoint of physics or hardware, if we place the antennas at random 
or at deterministic locations.
In particular, the current paper yields some important insights, which
cannot be inferred from~\cite{SF12}: We obtain a theoretical framework for 
radar operating with random antenna arrays, a technique which have been 
around for half a century; we show that Kerdock sequences, which are not 
useful for SISO or SIMO radar\footnote{SISO stands for
single-input-single-output radar, and SIMO for single-input-multiple-output
radar (i.e., a radar with one transmit and multiple receive antennas).}, 
are excellent for MIMO radar; our approach
allows for waveforms that satisfy a number of properties which are very
desirable in practice, and are not satisfied by random waveforms.
Indeed, as mentioned above, we also show that the finite harmonic 
oscillator system ``plays well'' with random antenna arrays.

\bigskip

Our paper is organized as follows. Section~\ref{s:setup} describes the
problem setup and the radar model. We review key properties of
Kerdock codes in Section~\ref{s:kerdock}. Our main theorem
is presented in Section~\ref{s:mainresult} and Section~\ref{s:proof} is
devoted to the proof of the main theorem. In Section~\ref{s:fos} we
extend our framework to other deterministic waveforms, such as
the finite harmonic oscillator system. Numerical simulations
are presented in Section~\ref{s:simulations}.

\subsection{Notation}

For a matrix $\vA$, we use $\vA^\ast$ to denote its adjoint matrix, which
is its conjugate transpose. The operator norm of $\vA$ is the largest
singular value of $\vA$ and is denoted by $\|\vA\|_{\op}$.
We denote the $k$-th column of $\vA$ by $\vA_k$ and
the element in the $i$-th row and $k$-th column by $\vA_{[i,k]}$.
The coherence of $\vA$ is defined as
\begin{equation}
\mu(\vA) := \underset{k \neq l}{\max} \,\,
\frac{|\langle \vA_k, \vA_l \rangle|}{\|\vA_k\|_2 \|\vA_l\|_2}.
\label{eq:coherence}
\end{equation}

The $n \times n$ Discrete Fourier Transform (DFT) matrix is written as
$\vF_n$ and the $n\times n$ identity matrix as $\vI_n$. For $\vx \in \CC^n$, let
$\vT_\tau$ denote the circulant translation operator, defined by
\begin{equation}
\label{translation}
\vT_\tau \vx(l) = \vx(l-\tau), \qquad \tau \in \CC^n,
\end{equation}
where $l-\tau$ is understood modulo $n$,
and let $\vM_f$ be the modulation operator defined by
\begin{equation}
\label{modulation}
\vM_f \vx(l) = \vx(l) e^{2\pi i fl/n}.
\end{equation}

\subsection*{Acknowledgements}

The authors acknowledge generous support by the National Science Foundation 
under grant DTRA-DMS 1042939 and by DARPA under grant N66001-11-1-4090.

\section{Problem Setup}
\label{s:setup}

We consider a MIMO radar employing $N_T$ antennas at the transmitter
and $N_R$ antennas at the receiver. We assume for convenience that
transmitters and receivers are co-located, cf.~Figure~\ref{fig:mimoradar}.
Furthermore, we assume a coherent propagation scenario, i.e.,
the element spacing is sufficiently small so that the radar return from
a given scatterer is fully correlated across the array.  The arrays and
all the scatterers are assumed to be in the same 2-D plane. The extension
to the 3-D case is straightforward.

The array manifolds $\aT$, $\aR$ with randomly spaced antennas are given by 
\begin{equation}
\label{transmit_array} 
\va_T(\beta) = 
\left[e^{2 \pi ip_1 \beta}, e^{2 \pi ip_2 \beta},\dots,
                       e^{2 \pi ip_{N_T} \beta}\right]^T,
\end{equation} 
and 
\begin{equation}
\label{receive_array} 
\va_R(\beta) = 
\left[ e^{2 \pi iq_1 \beta}, e^{2 \pi iq_2 \beta},\dots,
                       e^{2 \pi iq_{N_R} \beta}\right]^T,
\end{equation} 
where we assume that the relative antenna spacings $p_j$'s and $q_j$'s are 
i.i.d.\ uniformly on $[0,\frac{N_R N_T}{2}]$.
The $j$-th transmit antenna repeatedly transmits the signal $s_j(t)$,
which is assumed to be a periodic, continuous-time 
signal of period-duration $T$ seconds and bandwidth $B$. 
We observe the back-scattered signal over a duration $T$, and since
its bandwidth is $B$, it suffices that each receive antennas takes $N_s$ 
samples\footnote{Actually the received signal will have a somwhat larger
bandwidth $B_1 > B$ due to the Doppler effect. However, in practice this
increase in bandwidth is small, so we can assume $B \approx B_1$.}, 
where $N_s = T/\Delta_s$ and $\Delta_s = \frac{1}{2B}$.
It is convenient to introduce the finite-length vector
$\vs_j$
associated with $s_j$, via $\vs_j(l) := s_j(l \Delta_s), l = 1,\dots,N_s$.

Let $\vZ(t; \beta, \tau,f)$ be the $N_R \times N_s$ noise-free received
signal matrix from a unit strength target at direction $\beta$, delay $\tau$,
and Doppler $f$ (corresponding to its radial velocity with respect 
to the radar). Then
\begin{equation}
\vZ(t; \beta, \tau,f) = \va_R(\beta) \va_T^T(\beta) \vS^{T}_{\tau,f},
\notag
\end{equation}
where $\Std$ is a $N_s \times N_T$ matrix whose columns are the
circularly
delayed and Doppler shifted signals $s_j(t-\tau)e^{2 \pi if t}$.

We let $\vz(t; \beta, \tau, f) = {\rm vec}\{\vZ\}(t; \beta, \tau, f)$
be the noise-free vectorized received signal. 
We set up a discrete azimuth-range-Doppler grid $\{\beta_l, \tau_j, f_k\}$ for
$1\le l\le N_\beta$, $1\le j \le N_\tau$ and $1\le k\le N_f$,
where $\Delta_\beta, \Delta_\tau$ and $\Delta_f$ denote the corresponding 
discretization stepsizes. Using vectors $\vz(t; \beta_l,\tau_j,f_k)$ for 
all grid points $(\beta_l,\tau_j,f_k)$ we construct a complete response 
matrix $\vA$ whose columns
are $\vz(t; \beta_l,\tau_j,f_k)$ for $1 \le l \le  N_{\beta}$ and
$1 \le j \le  N_\tau$, $1 \le k \le N_f$. In
other words, $\vA$ is a $N_R N_s \times N_\tau N_{\beta} N_f$ matrix
with columns
\begin{equation}\label{dopplermatrix}
\vA_{\beta, \tau, f}=\va_R(\beta) \otimes \vS_{\tau, f}\va_T(\beta).
\end{equation}

Assume that the radar illuminates a scene consisting of $S$ scatterers
located on $S$ points of the $(\beta_l,\tau_j,f_k)$-grid. Let $\vx$ be a
sparse vector whose non-zero elements are the complex amplitudes of the
scatterers in the scene. The zero elements correspond to grid points
which are not occupied by scatterers. We can then define the radar
signal $\vy$ received from this scene by
\begin{equation}
\label{radarsystem}
\vy = \vA \vx + \vw
\end{equation}
where $\vy$ is an $N_R N_s\times 1$  vector,  $\vx$ is an  $N_\tau N_{\beta}
N_f \times 1 $ sparse vector and $\vw$ is an $N_R N_s\times 1$ complex Gaussian
noise vector.
Our goal is to solve for $\vx$, i.e., to locate
the scatterers (and their reflection coefficients) in the
azimuth-delay-Doppler domain.


\medskip
\noindent
{\bf Remark:}
The assumption that the targets lie on the grid points, while common
in compressive sensing, is certainly restrictive. A violation of
this assumption will result in a model mismatch, sometimes dubbed
{\em gridding error}, which can potentially be quite severe~\cite{HS10,CSP11}.
Recently some interesting strategies have been proposed to overcome 
this gridding error~\cite{FW12,TBS12}.
But these methods -- at least in their current form -- are not directly
applicable to our setting. This model mismatch issue is beyond the scope
of this paper and will be addressed in our future research.

\section{Kerdock codes}
\label{s:kerdock}

In this section we introduce one particularly useful set of
transmission waveforms.
Due to the setup in Section~\ref{s:setup} it suffices that we deal
with discrete, finite-length signals as transmission waveforms.
We briefly review the construction of Kerdock codes and some of their
fundamental properties.
There is a long list of properties that radar waveforms should satisfy.
As we will see in this paper, Kerdock codes fulfill many of them.
Kerdock codes over $\ZZ_2$ (i.e., binary Kerdock codes) were originally 
introduced in \cite{K72}. In the seminal paper~\cite{CCK97} the authors 
extend Kerdock codes from $\ZZ_2$ to $\ZZ_4$. By doing so, they uncover
many fascinating
properties of Kerdock codes and reveal numerous deep connections between
coding theory, discrete geometry and group theory. In the same paper,
the authors also extend Kerdock codes to the setting of $\ZZ_p$, 
where $p$ is an odd prime. 

Kerdock codes are an example of so-called mutually unbiased 
bases~\cite{WF89,SH03}.
Kerdock codes have also been proposed for use in communications
engineering~\cite{HSP06,IH08}. In~\cite{HCM06} the authors suggest the use
of Kerdock codes for radar, based on the peculiar
properties of the discrete ambiguity function associated with Kerdock
codes. We emphasize however that for the single transmit antenna radar scenario
Kerdock codes would actually perform rather badly, as discussed after Theorem
\ref{th:Kerdock} and shown by Figure \ref{fos1} in
 Section~\ref{s:simulations}. It is only in the setting of multiple
transmit antennas that Kerdock codes become useful for radar.

For the remainder of this paper we will only be concerned with Kerdock codes
over $\ZZ_p$. Some of the Kerdock codes over $\ZZ_p$, namely those 
corresponding to desarguesian planes in the language of~\cite{CCK97}, have 
also been derived earlier in~\cite{Lev82} and~\cite{Koe95}.
A simple way to construct these Kerdock codes is the following,
in which they arise as eigenvectors of time-frequency shift operators. 
Let $p$ be an odd prime number. For each $k=0,\dots,p-1$ we compute the 
eigenvector decomposition of $\vT_0 \vM_k$ (which always exists, since 
$\vT_0 \vM_k$ is a unitary matrix) 
\begin{equation}
\vU_{(k)} \boldsymbol{\Sigma}_{(k)} \vU_{(k)}^{\ast} = \vT_0 \vM_k,
\label{eigenvector1}
\end{equation}
where the unitary matrix $\vU_{(k)}$ contains the eigenvectors of 
$\vT_0 \vM_k$ and the diagonal matrix $\boldsymbol{\Sigma}_{(k)}$ the 
associated eigenvalues\footnote{The attentive reader will have noticed 
that $\vU_{(0)}$ is just the $p \times p$ DFT matrix $\vF_p$.}. 
Furthermore, we define $\vU_{(p)}: = I_p$. Now, let $\vu_{k,j}$ be the $j$-th 
column of $\vU_{(k)}$. The set consisting of the $p^2+p$ vectors 
$\{\vu_{k,j}, k=0,\dots,p; j=0,\dots,p-1\}$ 
forms a $\ZZ_p$-Kerdock code. 
There are numerous equivalent ways to derive
this Kerdock code, but, as pointed out earlier, not {\em all} Kerdock codes 
over $\ZZ_p$ are equivalent (see also the comment following Corollary 11.6
in~\cite{CCK97}). But we will be a bit sloppy, and simply refer to the
Kerdock code constructed above as {\em the} Kerdock code.

In the following theorem we collect those key properties of Kerdock codes 
that are most relevant for radar. These properties are either 
explicitly proved in~\cite{CCK97,HCM06} or can be derived easily 
from properties stated in those papers.

\begin{theorem} \label{th:Kerdock} 
Kerdock codes over $\ZZ_p$, where $p$ is an odd prime, satisfy the
following properties:
\begin{itemize}
\item[(i)] Mutually unbiased bases: For all $k=0,\dots,p$ and all
$j=0,\dots,p-1$, there holds:
$$
|\langle \vu_{k,j},\vu_{k',j'} \rangle| = 
\begin{cases}
1 & \text{if $k=k', j = j'$,} \\
0 & \text{if $k=k', j\neq j'$,} \\
\frac{1}{\sqrt{p}} & \text{if $k\neq k'$.}
\end{cases}
$$
\item[(ii)] Time-frequency ``autocorrelation'':\\
(a) For any fixed $(f,l)\neq(0,0)$ there exists a unique $k_0$ such that  
\begin{align}
 |\langle \vM_f \vT_l \vu_{k_0,j}, \vu_{k_0,j} \rangle| = 1 &
             \qquad  \text{for $j=0,\dots,p-1$,}\label{autocorr1} \\
 |\langle \vM_f \vT_l  \vu_{k,j}, \vu_{k,j}\rangle| = 0 & \qquad  \text{for $k\neq k_0$.}
\label{autocorr2}
\end{align}
(b) For any fixed $0\le k\le p-1$, there exist   $(f_r,l_r)$, $r=1,\dots, p$  such that  
\begin{align}
 |\langle \vM_{f_r} \vT_{l_r} \vu_{k,j}, \vu_{k,j} \rangle| = 1 &
             \qquad  \text{for $j=0,\dots,p-1$,}\label{autocorr3} 
\end{align}
\item[(iii)] Time-frequency crosscorrelation:
For all $k\neq k'$ and all $f$ and $l$ there holds:
\begin{equation}
\label{crosscorr}
|\langle \vM_f\vT_l \vu_{k,j} , \vu_{k',j} \rangle| \le \frac{1}{\sqrt{p}}  
  \qquad  \text{for $j=0,\dots,p-1$}.
\end{equation}
\item[(iv)] Polyphase property (Roots of unity property) in time and 
in frequency:\\
For any $k=0,\dots,p-1; j=0,\dots,p-1$, there holds:
\begin{equation}
\vu_{k,j}(l) = e^{2\pi i r/p} \quad \text{for some $r \in \{0,\dots,p-1\}$.}
\label{roots_time}
\end{equation}
For any $k=1,\dots,p; j=0,\dots,p-1$, there holds:
\begin{equation}
\hat{\vu}_{k,j}(l) = e^{2\pi i r/p} \quad \text{for some $r \in \{0,\dots,p-1\}$.}
\label{roots_frequency}
\end{equation}
\end{itemize}
\end{theorem}

\begin{proof}
Property~(i) is proved for instance in Lemma~11.3 in~\cite{CCK97}.
Properties~(ii) and~(iii) appear in Theorem~3 of~\cite{HCM06}.
Statement~\eqref{roots_time} of property~(iv) follows from the comment
right after Corollary~11.6 in~\cite{CCK97}.
Finally, statement~\eqref{roots_frequency} of property~(iv) follows 
from~\eqref{eigenvector1} together with property~(3) and the well-known 
fundamental relationships
$$\vF_p \vT_x \vF_p^{\ast} = \vM_{-x}, \qquad \vF_p \vM_x \vF_p^{\ast} = \vT_x.$$
\end{proof}

Kerdock codes have been proposed for adaptive radar in~\cite{HCM06}.
We emphasize again though that Kerdock codes would not be very effective
for a radar system with a single transmit antenna (SISO or SIMO radar). 
This can be easily seen as follows: Assume we only have one antenna that 
transmits one waveform $\vs$. Because of \eqref{autocorr3}, $\vs$ is (up to 
a constant phase factor) equal to $\vM_f\vT_l\vs$ for some $(f,l)$. In 
practice this ambiguity prevents us from determining the distance and the 
velocity of the object, when using Kerdock codes for SISO or SIMO.


As a consequence of the aforementioned ambiguity we will not use {\em
all} of the Kerdock codes as transmission signals for our MIMO radar,
instead we will choose one code for each index $k$. The reason is that
we need the waveforms to have low time-frequency crosscorrelation, 
 while \eqref{crosscorr} only holds when $k$ and $k'$ are different.

\begin{definition}[\bf{Kerdock waveforms}] Let $\{\vu_{k,j}, k=0,\dots,p,
j=0,\dots,p-1\}$ be a Kerdock code over $\ZZ_p$. The {\em Kerdock waveforms}
$\vk_0,\dots,\vk_r$, where $r < p$, are given by $\vk_k = \vu_{k,j}$ for
some arbitrary $j$. In other words, for each $k=0,\dots,r-1$ we pick an 
arbitrary vector from the orthonormal basis $\{\vu_{k,j}\}_{j=0}^{p-1}$.
\end{definition}
Note that Kerdock waveforms do not include any unit vectors, since only 
the first $r$ unitary matrices $\vU_{(0)},\dots, \vU_{(r-1)}$ are considered 
and $r$ is strictly less than $p$ (recall that $\vU_{(p)}=\vI_p$).

\section{The main result}
\label{s:mainresult}

As mentioned in the introduction, 
a standard approach to solve~\eqref{radarsystem} when $x$ is sparse, is
\begin{equation}
\label{Lasso}
\underset{\vx}{\min}\, \frac{1}{2}\|\vA\vx - \vy\|_2^2 + \lambda \|\vx\|_1,
\end{equation}
which is also known as lasso~\cite{T96}. 
But instead of~\eqref{Lasso}, we will use the {\em debiased lasso}.
That means first we compute an approximation $\tilde{I}$ for the support of
$\vx$ by solving~\eqref{Lasso}. This is the detection step. Then, in the
estimation step, we ``debias'' the solution by computing 
the amplitudes of $\vx$ via solving the reduced-size least squares problem
$\min \|\vA_{\tilde{I}} \vx_{\tilde{I}} - \vy\|_2$, where $\vA_{\tilde{I}}$
is the submatrix of $\vA$ consisting of the columns corresponding
to the index set $\tilde{I}$, and similarly for $\vx_{\tilde{I}}$.

We assume that the locations of the targets are random. To be precise,
we assume that the $S$ nonzero coefficients of $x$ are selected uniformly
at random and the phases of the non-zero entries of $x$ are random and
uniformly distributed in $[0, 2\pi)$. We will refer to this model
as the generic $S$-sparse model.

We are now ready to state our main result.

\begin{theorem} \label{th:maindoppler} 
Consider $\vy =
\vA \vx +\vw$, where $\vA$ is defined as in \eqref{dopplermatrix} and
$\vw_j \in {\mathcal CN}(0,\sigma^2)$.  
Assume that the positions of
the transmit and receive antennas  $p_j$'s and $q_j$'s are chosen
i.i.d. uniformly on $[0,\frac{N_RN_T}{2}]$ at random.
Suppose further that each transmit antenna sends a
different Kerdock waveform, i.e. the columns of the signal matrix $\vS$
are different Kerdock waveforms. 
Choose the discretization
stepsizes to be $\Delta_\beta = \frac{2}{N_R N_T}$,$\Delta_\tau = \frac{1}{2B}$, $\Delta_f = \frac{1}{T}$ and suppose that
\begin{equation}\label{assumptiondoppler} \max\big(N_RN_T, 32N_T^3\log
N_\tau N_fN_\beta\big)\le N_s=N_\tau , \end{equation} and also
\begin{equation}\label{assumptiondoppler2} \log^2N_\tau N_fN_\beta\le N_T\le N_R.  \end{equation}

If $\vx$ is drawn from the generic $S$-sparse scatterer model with
\begin{equation} S \le \frac{c_0 N_\tau}{\log N_\tau N_fN_\beta}
\label{lassosparsity1doppler} \end{equation} for some constant $c_0>0$,
and if \begin{equation} \underset{k \in I}{\min}\, |\vx_k| >
                   \frac{8\sqrt3\sigma}{\sqrt{N_R N_T}} \sqrt{2 \log
                   N_\tau N_fN_\beta},
\label{amplitudeproperty2} \end{equation} then the solution $\tilde{\vx}$
of the debiased lasso computed with $\lambda = 2 \sigma \sqrt{2
\log N_\tau N_fN_\beta}$ satisfies 
\begin{equation} \label{support2}
\supp (\tilde{\vx}) = \supp (\vx), 
\end{equation} 
with probability at
least 
\begin{equation} \label{probbound3} 1-p_1, 
\end{equation} 
and
\begin{equation} \frac{\|\tilde{\vx} - \vx \|_2}{\|\vx\|_2}
    \le \frac{5\sigma \sqrt{3N_RN_s}}{\|\vy\|_2}
\label{error2} 
\end{equation} 
with probability at least
\begin{equation} \label{probbound3a} (1-p_1)(1-p_2),
\end{equation} 
where 
\begin{align*} p_1 =16N_\tau^{-2}N_R^{-1}+ 8N_\tau^{-2}N_f^{-2}+
  4N_TN_\tau^{-2}N_f^{-2}+4(N_\tau N_f)^{-1}\\
 +4N_\tau^{-3}N_f^{-3} N_R^{-2} N_T^{-1}+8 N_T^{-2}( N_\tau N_fN_R)^{-3},
\end{align*} 
and 
$$ p_2 = 2(N_\tau N_fN_\beta)^{-1}(2\pi \log (N_\tau N_fN_\beta) + 
K(N_\tau N_fN_\beta)^{-1}) + {\mathcal O}((N_\tau N_fN_\beta)^{-2 \log 2}).$$ 
\end{theorem}

\medskip

\noindent
{\bf Remarks:}
\begin{enumerate}
\item The condition $N_T\le N_R$ in \eqref{assumptiondoppler2} is by no
means necessary, but rather to make our computation a little cleaner. We
could change it into $N_T\le 2N_R$, then the theorem would remain true with
a slightly different probability of success.  

\item It may seem that the conditions in \eqref{assumptiondoppler} and \eqref{assumptiondoppler2} are a bit restrictive. But, in practice, our method works with a broad range of parameters as the simulations show in Section~\ref{s:simulations}.

\end{enumerate}

\section{Proof of the result}
\label{s:proof}

To prove Theorem \ref{th:maindoppler}, we use a theorem by Cand\`es and Plan (Theorem~1.3 in~\cite{CP08}) which requires to estimate the operator norm of $\vA$ and the coherence of $\vA$. The original theorem only treats the real-valued case, it can be extended to complex-values case after some straightforward modifications (see Appendix B in \cite{SF12}).

\subsection{Auxiliary results} 

We first need the following Bernstein type lemma.
\begin{lemma}\label{quadlemma} Suppose $M$ is an $m\times m$ matrix,
$\alpha$ and $\beta$ are two joint independent random vectors in $\CC^m$
with zero means and $|\alpha_k|=|\beta_k|=1$ for $k=1,\dots, m$. If $n$ is a positive constant, 
then for any $t>0$ and $s>0$,
\begin{enumerate} \item if $|m_{kj}|\le\frac{1}{\sqrt{n}}$ for all $k,
j$, then

\begin{align}\label{Aabcase1} \Prob\Big( |\langle M\alpha,
\beta\rangle|\le mt \Big)\ge1-4m\exp\Big(-\frac{t^2}{4\frac{m}{n}}\Big).
\end{align} and \begin{align}\label{Aaacase1} \Prob\Big( |\langle M\alpha,
\alpha\rangle|\le 2mt \Big)\ge1-8m\exp\Big(-\frac{t^2}{2\frac{m}{n}}\Big),
\end{align} \item if $|m_{kj}|\le \frac{1}{\sqrt{n}}$ for $k\neq j$
and $m_{jj}=1$, then

\begin{align}\label{Aabcase2} \Prob\Big(
|\langle M\alpha, \beta\rangle|\le  s+mt
\Big)\ge1-4\exp\Big(-\frac{s^2}{4m}\Big)-4m\exp\Big(-\frac{t^2}{4\frac{m}{n}}\Big),
\end{align} and \begin{align}\label{Aaacase2} \Prob\Big(
m(1-2t)\le|\langle M\alpha, \alpha\rangle|\le m(1+2t)
\Big)\ge1-8m\exp\Big(-\frac{t^2}{2\frac{m}{n}}\Big).  \end{align}
\end{enumerate} \end{lemma}

\begin{proof} \begin{align*} \langle M\alpha,
\beta\rangle&=\sum_{k,j=1}^mm_{kj}\alpha_j\bar\beta_k\\
&=\sum_{l=1}^m\sum_{j=1}^mm_{j\oplus l,j}\alpha_j\bar\beta_{j\oplus l},
\end{align*}
 where $\oplus$ denotes addition modulo $m$.

 Let us first assume that $|m_{kj}|\le\frac{1}{\sqrt{n}}$.

Since $\alpha$ and $\beta$ are joint independent, then for any $l$, the
entries in $\sum_{j=1}^mm_{j\oplus l, j}\alpha_j\bar\beta_{j\oplus l}$
are all joint independent and it is easy to check that $\Exp(m_{j\oplus
l, j}\alpha_j\bar\beta_{j\oplus l})=0$ and $|m_{j\oplus l,
j}\alpha_j\bar\beta_{j\oplus l}|= |m_{j\oplus l, j}|$, then Theorem
4.5 in \cite{HRS12} will give,

\begin{align} \Prob\Big( |\sum_{j=1}^mm_{j\oplus
l, j}\alpha_j\bar\beta_{j\oplus l}|\le t \Big)&\ge
1-4\exp\Big(-\frac{t^2}{4\sum_{j}|m_{j\oplus l, j}|^2}\Big)\nonumber\\
&\ge1-4\exp\Big(-\frac{t^2}{4\frac{m}{n}}\Big).  \end{align}

We take all $m$ different choices of $l$,
then \begin{align}\label{quad1} \Prob\Big(
|\sum_{l=1}^m\sum_{j=1}^mm_{i\oplus l, j}\alpha_j\bar\beta_{j\oplus
l}|\le mt \Big)\ge1-4m\exp\Big(-\frac{t^2}{4\frac{m}{n}}\Big), \end{align}
which proves \eqref{Aabcase1}.

\begin{equation*} \langle M\alpha,
\alpha\rangle=\sum_{l=1}^m\sum_{j=1}^mm_{j\oplus l,
j}\alpha_j\bar\alpha_{j\oplus l}, \end{equation*} different from above,
the entries in $\sum_{j=1}^mm_{j\oplus l, j}\alpha_j\bar\alpha_{j\oplus
l}$ are no longer all jointly independent. But similar to the proof of
Theorem 5.1 in~\cite{PRT08} and Lemma 3 in \cite{SF12}, we observe that
for any $l$ we can split the index set ${1,\dots,m}$ into two subsets
$T_l^1,T_l^2\subset \{1,\dots,m\}$, each of size $m/2$, such that the
$m/2$ variables $\alpha_j\bar\alpha_{j\oplus l}$ are jointly independent
for $j\in T^1_l$, and analogous for $T^2_l$. (For convenience we assume
here that $m$ is even, but with a negligible modification the argument
also applies for odd $m$.) In other words, each of the sums $\sum_{j\in
T^r_l} m_{j\oplus l, j}\alpha_j\bar\alpha_{j\oplus l}, r=1,2$, contains
only jointly independent terms.

So for each $l$, \begin{equation} \Prob\Big(|\sum_{j\in T^r_l}
m_{j\oplus l, j}\alpha_j\bar\alpha_{j\oplus l}| \le t \Big)
\ge 1-4\exp\Big(-\frac{t^2}{2\frac{m}{n}}\Big), \end{equation}
which implies that \begin{align} \Prob\Big(|\sum_{j} m_{j\oplus
l, j}\alpha_j\bar\alpha_{j\oplus l}| \le 2t\Big) & \ge 1-8
\exp\Big(-\frac{t^2}{2\frac{m}{n}}\Big), \end{align}

Again, we take all $m$ different choices of $l$, then
\begin{align} \Prob\Big( |\sum_{l=1}^m\sum_{j=1}^mm_{j\oplus
l, j}\alpha_j\bar\alpha_{j\oplus l}|\le 2mt
\Big)\ge1-8m\exp\Big(-\frac{t^2}{2\frac{m}{n}}\Big), \end{align} which
proves \eqref{Aaacase1}.

Now let us assume that $|m_{kj}|\le \frac{1}{\sqrt{n}}$ for $k\neq j$
and $m_{jj}=1$.

\begin{align*} \langle M\alpha,
\beta\rangle&=\sum_{j=1}^mm_{jj}\alpha_j\bar\beta_{j}+\sum_{l=1}^{m-1}\sum_{j=1}^mm_{j\oplus
l, j}\alpha_j\bar\beta_{j\oplus l}\\
&=\sum_{j=1}^m\alpha_j\bar\beta_{j}+\sum_{l=1}^{m-1}\sum_{j=1}^mm_{j\oplus
l, j}\alpha_j\bar\beta_{j\oplus l}. \end{align*}

Since $\alpha$ and $\beta$ are joint independent and
$|\alpha_j\bar\beta_j|=1$, \begin{align}\label{diagonalest}
\Prob\Big( |\sum_{j=1}^m\alpha_j\bar\beta_{j}|\le s
\Big)\ge1-4\exp\Big(-\frac{s^2}{4m}\Big).  \end{align}

Similar to the proof of \eqref{quad1} above, we have that
\begin{align} \Prob\Big( |\sum_{l=1}^{m-1}\sum_{j=1}^mm_{j\oplus
l, j}\alpha_j\bar\beta_{j\oplus l}|\le (m-1)t
\Big)\ge1-4(m-1)\exp\Big(-\frac{t^2}{4\frac{m}{n}}\Big), \end{align}
together with \eqref{diagonalest}, it follows 
\begin{align} \Prob\Big( |\langle M\alpha, \beta\rangle|\le s+(m-1)t
\Big)\ge1-4\exp\Big(-\frac{s^2}{4m}\Big)-4(m-1)\exp\Big(-\frac{t^2}{4\frac{m}{n}}\Big),
\end{align} which proves \eqref{Aabcase2}.

\begin{equation*} \langle M\alpha,
\alpha\rangle=\sum_{j=1}^mm_{jj}+\sum_{l=1}^{m-1}\sum_{j=1}^mm_{j\oplus l,
j}\alpha_j\bar\alpha_{j\oplus l}=m+\sum_{l=1}^{m-1}\sum_{j=1}^mm_{j\oplus
l, j}\alpha_j\bar\alpha_{j\oplus l},  \end{equation*}
then \eqref{Aaacase2} results from similar proof as for \eqref{Aaacase1}
and the triangle inequality.  \end{proof}

\subsection{Estimation of the Operator Norm}

\begin{lemma}\label{th:normbounddoppler} Let $\vA$
be the matrix in Theorem \ref{th:maindoppler} satisfying \eqref{assumptiondoppler}. Then
\begin{equation} 
\Prob\Big(\|\vA\|^2_{\text{op}}\le 2N_fN_R^2N_T^2\Big)\ge 1-8N_\tau^{-2}
N_R^{-1}.  \end{equation} \end{lemma}

\begin{proof}

Since $\|\vA\|^2_{\text{op}}=\|\vA\vA^*\|_{\text{op}}$, we consider matrix
$\vB=\vA\vA^*$ as block matrix $$ \begin{bmatrix} \vB_{1,1}            &
\vB_{1,2}   & \dots      & \vB_{1,N_R} \\ \vdots             & \ddots    &
& \vdots    \\ \vB_{N_R,1}  &   \dots        &            & \vB_{N_R,N_R}
\end{bmatrix}, $$ where the blocks $\{\vB_{j,j'}\}_{j,j'=1}^{N_R}$
are matrices of size $N_t \times N_t$.

Via a simple permutation, we can turn $\vB$ into a matrix $\vC$ with
blocks $\{\vC_{l,l'}\}_{l,l'=1}^{N_s}$ of size $N_R\times N_R$, where
the $(j,j')$-th entry of the block $\vC_{l,l'}$ is defined as \begin{align}
&\vC_{[l,j;l'j']}=\vB_{[j,l;j',l']}  =  (\vA \vA^{\ast})_{[j,l;j',l']}
= \sum_{\beta} \sum_{\tau} \sum_{f}^{}\vA_{[j,l;\tau,f,\beta]}
\overline{\vA_{[j',l';\tau,f,\beta]}}\nonumber \\ &=\sum_{\beta} e^{2\pi i
(q_j-q_{j'})\beta}\sum_{k=1}^{N_T} \sum_{k'=1}^{N_T} e^{2\pi i
(p_k-p_{k'})\beta} \langle \vT_l\vk_k, \vT_{l'}\vk_{k'}\rangle\sum_{m=1}^{N_f}
e^{2\pi i(l-l')\Delta_tm\Delta_f }\nonumber\\ &=N_f\delta_{l,l'}\sum_{\beta}
e^{2\pi i(q_j-q_{j'})\beta}\sum_{k=1}^{N_T}
\sum_{k'=1}^{N_T} e^{2\pi i(p_k-p_{k'})\beta}\langle \vT_l\vk_k,
\vT_{l'}\vk_{k'}\rangle.
\end{align}

Then it is easy to see that $\vC$ is block-diagonal, and all the
diagonal-blocks are identical. So we only have to bound the first block
$\vC_{1,1}$.

\begin{align} &\vC_{[1,j;1,j']}=N_f\sum_{\beta} e^{2\pi i
(q_j-q_{j'})\beta}\sum_{k=1}^{N_T} \sum_{k'=1}^{N_T} e^{2\pi i
(p_k-p_{k'})\beta} \langle \vk_k, \vk_{k'}\rangle\nonumber\\
&=N_f\sum_{n=0}^{N_RN_T-1} e^{2\pi i(q_j-q_{j'})
\frac{n}{N_RN_T}}\sum_{k=1}^{N_T} \sum_{k'=1}^{N_T} e^{2\pi i(p_k-p_{k'})
\frac{n}{N_RN_T}} \langle \vk_k, \vk_{k'}\rangle.\nonumber\\
\end{align}

Define $c_n=\sum_{k=1}^{N_T}\sum_{k'=1}^{N_T} e^{2\pi i
(p_k-p_k')\frac{n}{N_RN_T}}\langle \vk_k, \vk_{k'}\rangle$, then
$$\vC_{1,1}=N_f\sum_{n=0}^{N_RN_T-1}c_nX_n,$$ where $X_n$ is the
matrix-valued random variable given by $(X_n)_{j,j'}=e^{2\pi i
(q_j-q_{j'})\frac{n}{N_RN_T}}$ and therefore $\|X_n\|_{\text{op}}=N_R$.


Note that  $\Exp(e^{2\pi i(p_k-p_{k'})n})=0$ and $|\langle \vk_k, \vk_{k'}\rangle|\le \frac{1}{\sqrt{N_s}}$ 
for $k\neq k'$. Choosing
$t=2\sqrt{\frac{N_T}{N_s}}\sqrt{\log N_\tau N_RN_T}$ in \eqref{Aaacase2}
of Lemma \ref{quadlemma}, we arrive at

$$\Prob\Big(|c_n|\le N_T(1+4\sqrt{\frac{N_T}{N_s}}\sqrt{\log N_\tau
N_RN_T})\Big)\ge 1-8N_T(N_\tau N_RN_T)^{-2},$$ then the assumption in
\eqref{assumptiondoppler} implies that $16N_T\log N_\tau
N_RN_T\le N_s$, therefore
$$\Prob\Big(|c_n|\le 2N_T\Big)\ge
1-8N_T(N_\tau N_RN_T)^{-2}.$$

We apply the union bound over the $N_R N_T$ possibilities associated
with $n$ and get $$\Prob\Big(\max |c_n|\le 2N_T\Big)\ge 1-8N_\tau^{-2}
N_R^{-1},$$ which implies that $$\Prob\Big(\|\vC_{1,1}\|_{\text{op}}\le
2N_fN_R^2N_T^2\Big)\ge 1-8N_\tau^{-2} N_R^{-1}.$$

Then the fact that $\|\vB\|_{\text{op}}=\|\vC\|_{\text{op}}=\|\vC_{1,1}\|_{\text{op}}$
will give us the conclusion.


\end{proof}

\subsection{Estimation of the Coherence}

\begin{lemma} \label{th:coherencedoppler} 
Let $\vA$ be the matrix in Theorem \ref{th:maindoppler} satisfying \eqref{assumptiondoppler} and \eqref{assumptiondoppler2}. Then 
\begin{equation} 
\underset{(\tau,\dopp,\beta)\neq (\tau',\dopp',\beta')}{\max} \big|\langle
\vA_{\tau,\dopp,\beta},\vA_{\tau',\dopp',\beta'} \rangle \big| \le
16N_R\log N_\tau N_fN_RN_T
\label{eq:coherencebounddoppler} 
\end{equation} 
with probability at
least 
$$1-8N_\tau^{-2}N_f^{-2}-4N_TN_\tau^{-2}N_f^{-2}-4(N_\tau
N_f)^{-1}-4N_\tau^{-3}N_f^{-3} N_R^{-2} N_T^{-1}-8 N_T^{-2}(
N_\tau N_fN_R)^{-3}.$$  
\end{lemma} 

\begin{proof} 
We need to find an upper bound for 
\begin{equation*} \max{|\langle \vA_{\tau,f,\beta},\vA_{\tau',f',\beta'} 
\rangle |} \qquad \text{for $(\tau,f,\beta) \neq (\tau',f',\beta')$}.  
\end{equation*} 
Recall the $\Std=\vM_f\vT_\tau \vS$, it follows from the definition that 
\begin{equation*} 
\vA_{\tau,\dopp,\beta} = \aR \otimes (\Std \aT), 
\end{equation*} 
from which we readily compute
\begin{equation*} 
|\langle \vA_{\tau,f,\beta},\vA_{\tau',f',\beta'} \rangle| = 
|\langle \aR, \aRp \rangle| |\langle \vS_{\tau,f} \aT , \vS_{\tau',f'} \aTp \rangle|.  
\end{equation*} 
We use the discretization $\beta = n \Delta_\beta$, $\beta' = n' \Delta_\beta$,
where $\Delta_\beta = \frac{2}{N_R N_T}$, $n,n' = 1,\dots,N_\beta$,
with $N_\beta = N_R N_T$.

Since \begin{equation*} 
|\langle \Std \aT , \Stdp \aTp \rangle| = |\langle \vS_{\tau - \tau',f-f'} \aT ,
\vS\aTp \rangle| 
\end{equation*} 
for $\tau,\tau'=0,\dots,N_\tau-1, f,f'=0,\dots,N_f-1$. 
We can confine the range of values for $\tau,\tau'$ to 
$\tau'=0, \tau=0,\dots,N_\tau-1$ and $f,f'$ to $f'=0, f=0,\dots, N_f-1$, 
then we only need to estimate $|\langle \Std \aT , \vS \aTp \rangle|$. 
We now consider three cases.

\noindent {\bf Case (i) $\beta\neq \beta', \tau =0, f=0$:}

By Theorem 4.5 in \cite{HRS12}, for any $t_1>0$ \begin{equation*} \Prob\Big( |\langle
\aR ,\aRp \rangle|\ge t_1 \Big) \le 4 \exp \Big(-\frac{t_1^2}{4N_R}\Big),
\end{equation*}
choosing $t_1=2\sqrt2\sqrt{N_R}\sqrt{\log N_\tau N_fN_RN_T}$ will give
us that \begin{equation}\label{Aabest} \Prob\Big(|\langle \aR ,\aRp
\rangle|\le 2\sqrt2\sqrt{N_R}\sqrt{\log N_\tau N_fN_RN_T} \Big)\ge
1-4(N_\tau N_fN_RN_T)^{-2}.  \end{equation}

Define $M=\vS^{\ast}\vS$ then $|m_{kj}|=|\langle \vk_k,
\vk_j\rangle|\le\frac{1}{\sqrt{N_s}}$ for $k\neq j$ and
$m_{jj}=1$. We choose $s=2\sqrt2\sqrt{N_T}\sqrt{\log N_\tau N_f
N_RN_T}$ and $t=2\sqrt2\sqrt{\frac{N_T}{N_s}}\sqrt{\log N_\tau
N_fN_RN_T}$ in \eqref{Aabcase2} of Lemma \ref{quadlemma} and get
\begin{align*} \Prob\Big( | \langle \vS^{\ast}\vS \aT ,\aTp \rangle|\le
(2\sqrt2\sqrt{N_T}+2\sqrt2 N_T \sqrt{\frac{N_T}{N_s}})\sqrt{\log N_\tau
N_fN_RN_T}\Big)\nonumber\\ \ge1-4(N_\tau N_fN_RN_T)^{-2}-4N_T(N_\tau
N_fN_RN_T)^{-2}, \end{align*} combined with \eqref{Aabest},
\begin{align*} \Prob\Big(|\langle \vA_{\tau,f,\beta}, \vA_{\tau,f',\beta'}
\rangle|\le 8(\sqrt{N_RN_T}+N_T \sqrt{\frac{N_RN_T}{N_s}})\log N_\tau
N_fN_RN_T  \Big)\nonumber\\ \ge 1-8(N_\tau N_fN_RN_T)^{-2}-4N_T(N_\tau
N_fN_RN_T)^{-2}.  \end{align*}

After taking the union bound over $(N_RN_T)^2$ different possibilities
associated with $\beta, \beta'$, we will have that

\begin{align*} \Prob\Big(\max|\langle \vA_{\tau,f,\beta},
\vA_{\tau,f',\beta'} \rangle|\le 8(\sqrt{N_RN_T}+N_T
\sqrt{\frac{N_RN_T}{N_s}})\log N_\tau N_fN_RN_T  \Big)\nonumber\\ \ge
1-8N_\tau^{-2}N_f^{-2}-4N_TN_\tau^{-2}N_f^{-2}.  \end{align*}

A little algebra, using \eqref{assumptiondoppler} and \eqref{assumptiondoppler2},
shows that 
$$8(\sqrt{N_RN_T}+N_T
\sqrt{\frac{N_RN_T}{N_s}})\le 16N_R,$$
therefore
\begin{align}\label{eq:case1} \Prob\Big(
\max|\langle \vA_{\tau,f,\beta},\vA_{\tau,f,\beta'}\rangle| \le
16N_R\log N_\tau N_fN_RN_T
\Big)\nonumber\\ \ge  1-8N_\tau^{-2}N_f^{-2}-4N_TN_\tau^{-2}N_f^{-2}.
\end{align} \medskip \noindent {\bf Case (ii) $\beta\neq \beta', (\tau,
f) \neq (0, 0)$:}

For the same reason, here holds \eqref{Aabest}.

Define $C= \vS_{\tau,f}^{\ast}\vS$, from the properties of $\vk_j$, 
we have that $$|c_{kj}|=|\langle \vM_f\vT_\tau \vk_k, \vk
_j\rangle|\le\frac{1}{\sqrt{N_s}}\  \text{for} \ k\neq j$$ and
there exists $j_0$ such that $|c_{j_0j_0}|=1$ and $c_{jj}=0$
for $j\neq j_0$. Then $$|\langle\vS_{\tau,f}^{\ast}\vS\aT,\aTp
\rangle|\le1+|\langle C'\aT,\aTp \rangle|$$ where $C' $ is a
zero-diagonal matrix which coincides $C$ at off-diagonal
entries. Certainly $C'$ satisfies the condition for \eqref{Aabcase1}
to hold. Choosing $t=4\frac{\sqrt{N_T}}{\sqrt{N_s}}\sqrt{\log
N_\tau N_fN_RN_T} $ in \eqref{Aabcase1} of Lemma
\ref{quadlemma} yields \begin{align}\label{betabeta}
\Prob\Big( |\langle\vS_{\tau,f}^{\ast}\vS\aT,\aTp \rangle |\le
1+4\frac{N_T\sqrt{N_T}}{\sqrt{N_s}}\sqrt{\log N_\tau N_fN_RN_T}|
\Big)\nonumber\\ \ge 1-4 N_T( N_\tau N_fN_RN_T)^{-4}, \end{align}
from the assumption that $32N_T^3\log N_\tau N_fN_RN_T\le N_s$,
together with \eqref{Aabest}, we will get \begin{align*} \Prob\Big(
|\langle \vA_{\tau,f,\beta},\vA_{\tau',f',\beta'}\rangle| \le
4\sqrt2\sqrt{N_R}\sqrt{\log N_\tau N_fN_RN_T}
\Big)\nonumber\\ \ge 1-4(N_\tau N_fN_RN_T)^{-2}-4 N_T( N_\tau
N_fN_RN_T)^{-4}.  \end{align*}

By \eqref{assumptiondoppler2}, we deduce $\log N_\tau N_fN_RN_T\le N_R$. Therefore

 \begin{align*} \Prob\Big(
|\langle \vA_{\tau,f,\beta},\vA_{\tau',f',\beta'}\rangle| \le
4\sqrt2N_R\Big)\nonumber \ge 1-4(N_\tau N_fN_RN_T)^{-2}-4 N_T( N_\tau
N_fN_RN_T)^{-4}.  \end{align*}

We apply the union bound over $N_\tau N_fN_R^2N_T^2$ possibilities and
arrive at

\begin{align}\label{eq:case2} 
\Prob\Big(\max|\langle
\vA_{\tau, f, \beta}, \vA_{\tau', f', \beta'} \rangle|\le
4\sqrt{2}N_R\Big) \ge 1-4(N_\tau N_f)^{-1}-4N_\tau^{-3}N_f^{-3} N_R^{-2}
N_T^{-1}.  \end{align}

\medskip \noindent {\bf Case (iii) $\beta  = \beta',(\tau, f) \neq
(0, 0)$:}

Note that the matrix $C= \vS_{\tau,f}^{\ast}\vS$ has exactly the same
properties as in Case (ii) above. Following the same argument as we show
\eqref{betabeta} and applying \eqref{Aaacase1} of Lemma \ref{quadlemma}
combined with the assumption as in \eqref{assumptiondoppler} gives us that

\begin{align*} \Prob\Big( |\langle\vS_{\tau,f}^{\ast}\vS\aT,\aT
\rangle |\le 1+4\sqrt{2}\frac{N_T\sqrt{N_T}}{\sqrt{N_s}}\sqrt{\log
N_\tau N_fN_RN_T} \Big)\\ \ge 1-8 N_T( N_\tau N_fN_RN_T)^{-4},
\end{align*} which implies that \begin{align*} \Prob\Big(
|\langle \vA_{\tau,f,\beta},\vA_{\tau',f',\beta'}\rangle| \le
2N_R \Big) \ge 1-8 N_T^{-3}( N_\tau N_fN_R)^{-4}, \end{align*}

We apply the union bound over the $N_\tau N_fN_R N_T$ possibilities
associated with $\tau$, $f$ and $\beta$

\begin{equation} \label{eq:case3} \Prob\Big(\max |\langle
\vA_{\tau,f,\beta}, \vA_{\tau',f',\beta} \rangle| \le
2N_R\Big) \ge1-8N_T^{-2}(N_\tau N_f N_R)^{-3}.  \end{equation}
\eqref{eq:case1}, \eqref{eq:case2} and \eqref{eq:case3} will give the
conclusion.

\end{proof}

\subsection{Regarding the matrix with unit-norm columns}

In order to apply Theorem~1.3 in~\cite{CP08}, we need to normalize the columns
of $\vA$. We first have the following result which shows the lower and upper bounds of the
norm of columns of $\vA$.

\begin{lemma}\label{columnnorm} Let $\vA$ be
defined as in Theorem \ref{th:maindoppler} satisfying \eqref{assumptiondoppler}, then
\begin{align} \Prob\Big(\frac{1}{3}N_RN_T\le \min \|\vA_{\tau,f,\beta}\|_2^2\le\max 
\|\vA_{\tau,f,\beta}\|_2^2\le\frac{5}{3}N_RN_T
\Big)\ge1-8N_\tau^{-2}N_R^{-1}.  \end{align}
\end{lemma} \begin{proof} Recall that \begin{align*}
\|\vA_{\tau,f,\beta}\|_2^2 = \|\aR\|_2^2 \|\vS_{\tau,f} \aT\|_2^2&=N_R
\langle\vS_{\tau,f}^\ast\vS_{\tau,f} \aT, \aT \rangle\\ &=N_R \langle
\vS^\ast \vS \aT, \aT \rangle.  \end{align*}

Setting $t=2\sqrt{\frac{N_T}{N_s}}\sqrt{\log N_\tau N_RN_T}$ in
\eqref{Aaacase2} of Lemma \ref{quadlemma} yields \begin{align*}
\Prob\Big(N_T(1-4\sqrt{\frac{N_T}{N_s}}\sqrt{\log N_\tau N_RN_T})\le
| \langle \vS^{\ast}\vS \aT ,\aT \rangle|\le  \\ 
N_T(1+4\sqrt{\frac{N_T}{N_s}}\sqrt{\log N_\tau N_RN_T})\Big)
\ge1-8N_T(N_\tau N_RN_T)^{-2}.  \end{align*}

An easy calculation from \eqref{assumptiondoppler} leads
$$4\sqrt{\frac{N_T}{N_s}}\sqrt{\log N_\tau N_RN_T}\le\frac{2}{3},$$
which indeed implies \begin{align} \Prob\Big( \frac{1}{3}N_T \le| \langle \vS^{\ast}\vS
\aT ,\aT \rangle|\le  \frac{5}{3}N_T\Big)\ge1-8N_T(N_\tau N_RN_T)^{-2}.
\end{align} Since the above probability does not depend on $\tau$ or $f$,
we take all $N_RN_T$ possibilities of $\beta$ and conclude the proof of
the lemma.  \end{proof}

\begin{corollary} Suppose $\tilde{\vA}=\vA\vD^{-1}$ where $\vD$
is the $N_\tau N_fN_\beta\times N_\tau N_fN_\beta$ diagonal matrix
defined by $\vD_{(\tau,f,\beta),(\tau,f,\beta)} = \|\vA_{\tau,
f,\beta}\|_2$, or in other words $\tilde\vA$ is the matrix with
unit-norm columns from $\vA$. Then \begin{equation}\label{tildenorm}
\Prob\Big(\|\tilde\vA\|^2_{\text{op}}\le  6N_fN_R
N_T\Big)\ge1-16N_\tau^{-2}N_R^{-1},  \end{equation} and
\begin{equation}\label{tildecoh} \Prob \Big(\mu\big(\tilde{\vA}\big) \le
  48\frac{\log N_\tau N_fN_RN_T}{N_T}\Big)
\ge 1-p_3, \end{equation} where \begin{align*}
p_3=8N_\tau^{-2}N_R^{-1}+8N_\tau^{-2}N_f^{-2}+4N_TN_\tau^{-2}N_f^{-2}+4(N_\tau
N_f)^{-1}\\ +4N_\tau^{-3}N_f^{-3} N_R^{-2} N_T^{-1}+8 N_T^{-2}( N_\tau
N_fN_R)^{-3}.  \end{align*} \end{corollary}

\begin{proof} This corollary is a direct consequence of Lemma
\ref{th:normbounddoppler}, Lemma \ref{th:coherencedoppler} and Lemma
\ref{columnnorm}.  \end{proof}

\subsection{Assembling the proof of Theorem~\ref{th:maindoppler}}  \label{ss:assemblingproof}

\medskip \noindent \begin{proof}{(of Theorem~\ref{th:maindoppler})} Recall that we are trying to use 
Theorem~1.3 in~\cite{CP08} to prove Theorem~\ref{th:maindoppler}. We thus
need to verify that all assumptions of that theorem are satisfied.

We first point out that the assumptions of Theorem~\ref{th:maindoppler}
imply that the conditions of Lemma~\ref{th:normbounddoppler} and
Lemma~\ref{th:coherencedoppler} are fulfilled.

Note that solution $\tilde{\vx}$ of~\eqref{Lasso} and the solution
$\tilde{\vxt}$ of the following lasso problem \begin{equation}
\underset{\vxt}{\min}\, \frac{1}{2}\|\vA \vD^{-1}\vxt - \vy\|_2^2 +
\lambda \|\vxt\|_1, \qquad \text{with}\,\,\lambda = 2 \sigma\sqrt{2
\log(N_\tau N_R N_T)}, \label{lassoDI} \end{equation} are related by
$\tilde{\vx} = \vD^{-1}\tilde{\vxt}$.

We will first establish the claims in Theorem~\ref{th:maindoppler}
for the system $\vAt \vxt = \vy$  where  $\vAt = \vA \vD^{-1}$, $\vxt =
\vD\vx$.

First, the assumption~\eqref{amplitudeproperty2} and the fact that $\vz =
\vD \vx$ imply that \begin{equation} |z_k| \ge \frac{8\sqrt{3}
\min\|\vA_{\tau,f,\beta}\|_2}{\sqrt{N_R N_T}} \sigma\sqrt{2\log
(N_\tau N_fN_\beta)}\ge 8 \sigma \sqrt{2\log (N_\tau N_fN_\beta)},
\qquad \text{for $k \in S$,} \label{zcond1} \end{equation} with
probability at least $1-8N_\tau^{-2}N_R^{-1}$,thus establishing
the first condition of Theorem~1.3 in~\cite{CP08}.

Using the assumptions in Theorem \ref{th:maindoppler}, 
and the coherence bound~\eqref{tildecoh} we
compute \begin{equation}\label{normalcoh}\mu(\vAt) \le 48\frac{\log N_\tau N_fN_RN_T}{N_T}
\le  \frac{48}{\log N_\tau N_fN_\beta},\end{equation} which
holds with probability as in~\eqref{tildecoh}, and thus the coherence
property (1.5) in~\cite{CP08} is fulfilled.

Furthermore, using~\eqref{tildenorm} we see that
condition~\eqref{lassosparsity1doppler} implies \begin{equation}
S \le \frac{c_0 N_\tau}{ \log (N_\tau N_fN_\beta)} \le \frac{6c_0
N_\tau N_fN_\beta}{\|\vAt\|_{\op}^2 \log (N_\tau N_fN_\beta)}
\label{lassosparsity2} \end{equation} with probability at least
$1-16N_\tau^{-2}N_R^{-1}.$ Thus the sparsity assumption of Theorem~1.3 in~\cite{CP08}
 is also fulfilled and we obtain that
\begin{equation} \label{support4} \supp (\tilde{\vxt}) = \supp
(\vxt), \end{equation} with probability at least $1-p_1$. We note
that the relation $\supp(\tilde{\vx}) = \supp(\vx)$ holds with the
same probabiltity as the relation $\supp(\tilde{\vxt}) = \supp(\vxt)$, since $\supp(\vxt) = \supp(\vx)$
and multiplication by an invertible diagonal matrix does not change
the support of a vector. This establishes~\eqref{support2} with the
corresponding probability.

Once we have recovered the support of $\vx$, call it $I$, we can solve 
for the coefficients of $\vx$ by solving the standard least
squares problem $\min \|\vA_I \vx_I - \vy\|_2$, where $\vA_I$ is tbe
submatrix of $\vA$ whose columns correspond to the support set $I$,
and similarly for $\vx_I$. Note that the proof
of Theorem~3.2 in~\cite{CP08} yields as side result that 
with high probability the eigenvalues of any submatrix $\vA_I^{\ast} \vA_I$ 
with $|I| \le S$ are contained in the interval $[1/2, 3/2]$, 
which of course implies that $\kappa(\vA_I) \le \sqrt{3}$. By substituting this 
bound into the standard error bound, ~(5.8.11) in~\cite{HJ90}, we have that
\begin{equation} \frac{\|\tilde{\vxt} - \vxt \|_2}{\|\vxt\|_2}
    \le \frac{\sigma \sqrt{3N_R N_s}}{\|\vy\|_2}
\label{error4} \end{equation} which holds with probability at least
\begin{equation} \big(1 - p_1)(1 - p_2\big).  \end{equation}

Using the fact that $\tilde{\vz} = \vD \tilde{\vx}$, we compute $$
\frac{1}{\kappa(\vD)} \frac{\|\tilde{\vx} - \vx \|_2}{\|\vx\|_2} \le
\frac{\|\vD(\tilde{\vx} - \vx) \|_2}{\|\vD\vx\|_2}  = \frac{\|\tilde{\vxt}
- \vxt \|_2}{\|\vxt\|_2}, $$ or, equivalently, \begin{equation}
 \frac{\|\tilde{\vx} - \vx \|_2}{\|\vx\|_2} \le
\kappa(\vD) \frac{\|\tilde{\vxt} - \vxt
\|_2}{\|\vxt\|_2}.
\label{z2x} \end{equation} The bound~\eqref{error2} follows now from
combining~\eqref{error4}, ~\eqref{z2x} with the fact that $\kappa(\vD)\le 5$ (from Lemma~\ref{columnnorm}) .

\end{proof}

\section{Extension of the main result: Beyond Kerdock codes}
\label{s:fos}

In this section, we present a modified version of Theorem~\ref{th:maindoppler} 
that applies to waveforms that satisfy slightly more restrictive incoherence 
conditions. As such, Theorem~\ref{th:new} below does not hold for Kerdock 
waveforms, but the advantage compared to Theorem~\ref{th:maindoppler} is 
that the result also applies to radar systems with only one transmit antenna.

\begin{theorem} \label{th:new} 
Consider $\vy =
\vA \vx +\vw$, where $\vA$ is defined as in \eqref{dopplermatrix} and
$\vw_j \in {\mathcal CN}(0,\sigma^2)$. 
Suppose the transmission waveforms $\vs_j$'s satisfy the following conditions
\begin{align}
\label{coherence_cond1}
|\langle \vs_j, \vM_f \vT_\tau\rangle & \vs_j|\le
\frac{\gamma}{\sqrt{p}} \ \ \textrm{for} \ \ (f, \tau)\neq(0,0), \\
\label{coherence_cond2}
|\langle \vs_k, \vM_f \vT_\tau\rangle & \vs_j|\le
\frac{\gamma}{\sqrt{p}} \ \ \textrm{for} \ \  k \neq j,
\end{align}
where $\gamma>0$ is a fixed constant. Assume that the positions of
the transmit and receive antennas  $p_j$'s and $q_j$'s are chosen
i.i.d. uniformly on $[0,\frac{N_RN_T}{2}]$ at random. Choose the discretization
stepsizes to be $\Delta_\beta = \frac{2}{N_R N_T}$,$\Delta_\tau = \frac{1}{2B}$, $\Delta_f = \frac{1}{T}$
 and suppose that
\begin{equation}\label{newcondi1} \max\big(\gamma^2N_RN_T, 16\gamma^2N_T\log^3
N_\tau N_fN_\beta\big)\le N_s=N_\tau , \end{equation} and also
\begin{equation}\label{newcondi2}\gamma^2N_T\log^4 N_\tau N_fN_\beta\le N_sN_R,\ \ \   \log^2N_\tau N_fN_\beta\le N_T\le N_R.  \end{equation}
Then if the rest of the conditions of Theorem \ref{th:maindoppler} hold, we have the same conclusion as in Theorem \ref{th:maindoppler}. 
\end{theorem}

\begin{proof}
The proof of this theorem is similar to the one of Theorem \ref{th:maindoppler}. The main difference will arise when estimating the coherence of the matrix $\vA$. \eqref{eq:case1} will remain the same except with an extra $\gamma$ factor. From the conditions of the waveforms above, \eqref{eq:case2} becomes 

\begin{align*}
\Prob\Big(\max|\langle
\vA_{\tau, f, \beta}, \vA_{\tau', f', \beta'} \rangle|\le
8\sqrt{2}\frac{\gamma N_T\sqrt{N_TN_R}}{\sqrt{N_s}}\log N_\tau N_fN_RN_T\Big)\\
 \ge 1-4(N_\tau N_f)^{-1}-4N_\tau^{-3}N_f^{-3} N_R^{-2}
N_T^{-1},  \end{align*}
and we have to change \eqref{eq:case3} in the same manner. 

Then from the new conditions \eqref{newcondi1} and \eqref{newcondi2}, we
will have a similar estimate for the coherence of the normalized matrix $\vAt$ as in \eqref{normalcoh}.
\end{proof}

There are several examples of signal sets that satisfy the above conditions.
Perhaps the most intriguing example is the finite harmonic oscillator system (FHOS) 
constructed in \cite{GHS08}. This signal set in $\CC^p$ (where $p$ is a 
prime number) of cardinality $\mathcal{O}(p^3)$
satisfies~\eqref{coherence_cond1} and~\eqref{coherence_cond2}
with $\gamma=4$. An elementary construction of the FHOS for prime number 
$p\ge5$ can be found in \cite{WG11}. We illustrate the performance
of the FHOS and its comparison to Kerodck codes in numerical simulations 
in the next section.

\section{Numerical simulations}
\label{s:simulations}

In this section we will demonstrate the performance of our algorithms via
numerical simulations. We use the Matlab Toolbox TFOCS (\cite{BCG11}) and
choose in TFOCS Auslender and Teboulle's single-projection method to
solve~\eqref{Lasso}. The main computational costs per iteration of this
method are the operations $\vA x$ and $\vA^* y$. One can of course make 
$\vA$ explicit and do the regular matrix multiplication, but due to the 
special structure of $\vA$, we make the following observation to 
accelerate the computation. 

Recall that 
\begin{equation*}
\vA_{\tau,f, \beta} = \aR \otimes (\vS_{\tau,f} \aT).
\end{equation*}

First suppose we have $x$ given and we want to compute $y=\vA x$. Here $y$ is an $N_RN_s\times 1$ vector and $x$ is an $N_\tau N_fN_\beta\times 1$ vector. Instead of doing the direct matrix-vector multiplication, we divide $y$ into $N_R$ blocks $y_j$, each of which is of size $N_s\times 1$. Then
\begin{align*}
y_j&=\sum_{\tau,\beta,f} \aR_j \vS_{\tau,f}\aT x(\tau,\beta,f)\\
&=\sum_\beta \aR_j\sum_{\tau,f}\vS_{\tau,f}\aT x(\tau,\beta,f).
\end{align*}
Since for any fixed $\beta$ and $f$, we can consider
$x_{\beta,f}(\tau)=x(\tau,\beta,f)$ as an $N_\tau\times 1$ vector. Then,
recalling that $\vS_{\tau,f}\aT=\vM_f \vT_\tau (\vS \aT)$, an easy
observation yields 
\begin{align*}
y_j &=\sum_\beta  \aR_j\sum_f \vM_f\Big(\sum_\tau \vT_\tau\vS\aT x(\tau,\beta,f)\Big)\nonumber \\
&=\sum_\beta  \aR_j\sum_f \vM_f\Big(\vS\aT\ast x_{\beta,f}\Big),
\end{align*}
where the convolution can be implemented via FFT.

Now we suppose that $y$ is given and we want to compute $x=\vA^\ast y$. Note that in this case we have that the row vector of $\vA^\ast$ is of the form
$$\vA_{\tau,f,\beta}^\ast = \aR^\ast \otimes (\vS_{\tau,f} \aT)^\ast.$$

We divide $x$ into $N_\beta N_f$ blocks $x_{\beta,f}$, each of which is of size $N_\tau$. We also divide $y$ into $N_R$ blocks $y_j$, each of which is of size $N_s$. 
 $$x_{\beta,f}=\sum_{j=1}^{N_R} \overline{\aR_j} \vC_{\beta,f} y_j,$$
where $\vC_{\beta,f}$ is the matrix whose rows are  $(\vS_{\tau,f}\aT)^\ast=(\vM_f\vS_\tau\aT)^\ast$, an easy calculation leads us to 
$$x_{\beta,f}=\sum_{j=1}^{N_R} \overline{\aR_j} \vB_{\beta} (\vM_fy_i),$$
where $\vB_\beta$ is the matrix whose rows are $(\vS_\tau\aT)^\ast$. So $\vB_\beta$ is a circulant matrix and its first column $c_\beta$ is that $c_\beta(1)=\overline{\vS\aT(1)}$ and $c_\beta(k)=\overline{\vS\aT(N_s-k)}$ for $k=2, \dots, N_s$. Then we have $\vB_\beta( \vM_f y_j)=c_\beta \otimes \vM_f y_j$, which implies 
 $$x_{\beta,f}=\sum_{j=1}^{N_R} \overline{\aR_j} (c_\beta\otimes \vM_fy_j).$$

In each experiment, the locations of the transmit and receive antennas are chosen i.i.d. randomly on $[0, \frac{N_RN_T}{2}]$ and $S$ scatterers are placed randomly on the range-azimuth-Doppler grid, i.e the vector $x$ has $S$ entries at random locations along the vector. White Gaussian
noise is added to the composite data vector $\vA x$ with variance $\sigma^2$ determined to produce
the specified output signal-to-noise ratio. The lasso solution $\hat{x}$ is calculated with $\lambda$ as specified in Theorem \ref{th:maindoppler}. 
The experiment is repeated 50 times. 
Each experiment uses independent noise realization.

The probabilities of detection $P_d$ and false alarm $P_{fa}$ are computed as follows. The values of
the estimated vector $\hat{x}$ corresponding to the true scatterer locations are compared to a threshold.
Detection is declared whenever a value exceeds the threshold. The probability of detection is
defined as the number of detections divided by the total number of scatterers $S$. Next the values
of the estimated vector $\hat{x}$ corresponding to locations not containing scatterers are compared to
the same threshold. A false alarm is declared whenever one of these values exceeds the threshold. The
probability of false alarm is defined as the number of false alarms divided by $n-S$, where $n$ is the signal dimension. The probabilities of detection and false alarm are averaged over the 50 repetitions
of the experiment.

The probabilities are computed for a range of values of the threshold to produce the so-called
Receiver Operating Characteristics (ROC) [14, 28, 25] - the graph of $P_d$ vs. $P_{fa}$. As the threshold
decreases, the probability of detection increases and so does the probability of false alarm. In
practice the threshold is usually adjusted to achieve a specified probability of false alarm.
We note that the probability of detection increases as the SNR increases and decreases as $S$, the
number of scatterers increases.

\begin{figure}[h!]
\centering
\begin{tabular}{c}
  \includegraphics[width=80mm]{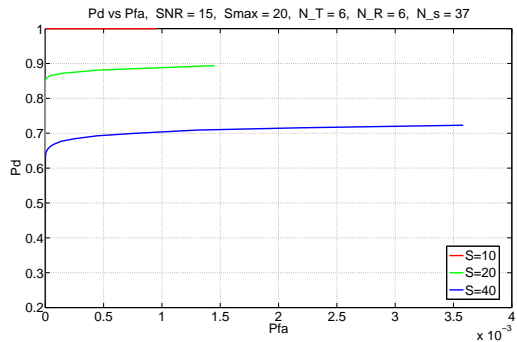} 
 
\end{tabular}

 \caption{MIMO, Kerdock, SNR=15} \label{SNR=15}
\end{figure}

\begin{figure}[h!]
\centering
\begin{tabular}{c}
  \includegraphics[width=80mm]{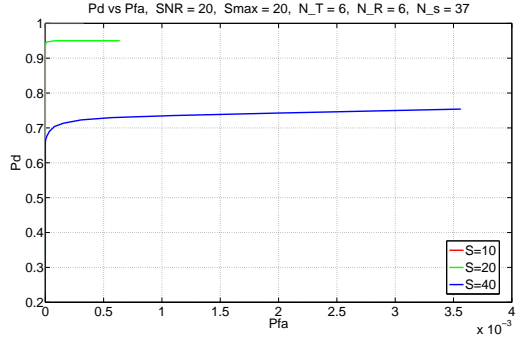} 
 
\end{tabular}

  \caption{MIMO, Kerdock codes, SNR=20} \label{SNR=20}
\end{figure}

We carry out two sets of simulations with different parameters to show the performance of the algorithms.
\begin{enumerate}
\item The first set of simulations is done using Kerdock codes as
transmission waveforms. The following parameters are used: $N_T = 6, N_R =
6, N_s = 37, N_f=37$ (hence we are using Kerdock waveforms of length 37); $S_{\text{max}}=20$ and the actual number of targets
is $S=S_{\text{max}}/2, S_{\text{max}}, 2S_{\text{max}}$, while the SNR is
chosen to be 15dB, 20dB and 25dB in Figure \ref{SNR=15}, Figure \ref{SNR=20} and Figure \ref{SNR=25} respectively. This set of simulations is aimed to
demonstrate the efficiency of Kerdock codes for MIMO radar at different SNR 
levels. 

\item The second set of simulations is to compare the Kerdock codes and the
finite harmonic oscillator system for both SIMO and MIMO. We fix the following
parameters in this set of simulations:  $S_{\text{max}}=10$ and
$S=S_{\text{max}}/2, S_{\text{max}}, 2S_{\text{max}}$, while the SNR is
fixed to be 15dB. We choose  $N_T = 1, N_R = 8, N_s = 11, N_f=11$ for SIMO and  $N_T = 2, N_R = 8, N_s = 17, N_f=17$ for MIMO. As we mentioned before, Figure \ref{fos1} shows that Kerdock codes are not very good choice in SIMO. But Kerdock codes are already very efficient when $N_T=2$ as shown in Figure \ref{fos2}, which also shows that the condition $\log^2N_\tau N_fN_\beta\le N_T$ in \eqref{assumptiondoppler2} is a restrictive theoretical condition and we can do much better in practice.

\end{enumerate}

\begin{figure}[h!]
\centering
\begin{tabular}{c}
  \includegraphics[width=80mm]{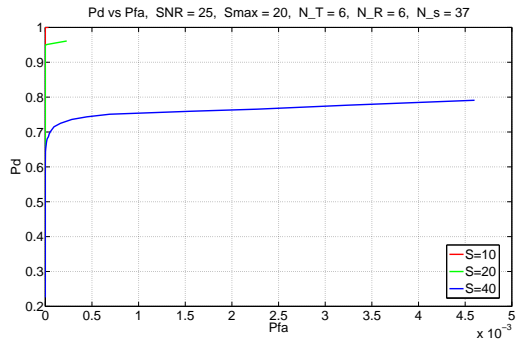} 
 
\end{tabular}

  \caption{MIMO, Kerdock codes, SNR=25} \label{SNR=25}
\end{figure}

\begin{figure}[h!]
\centering
\begin{tabular}{c}
  \includegraphics[width=80mm]{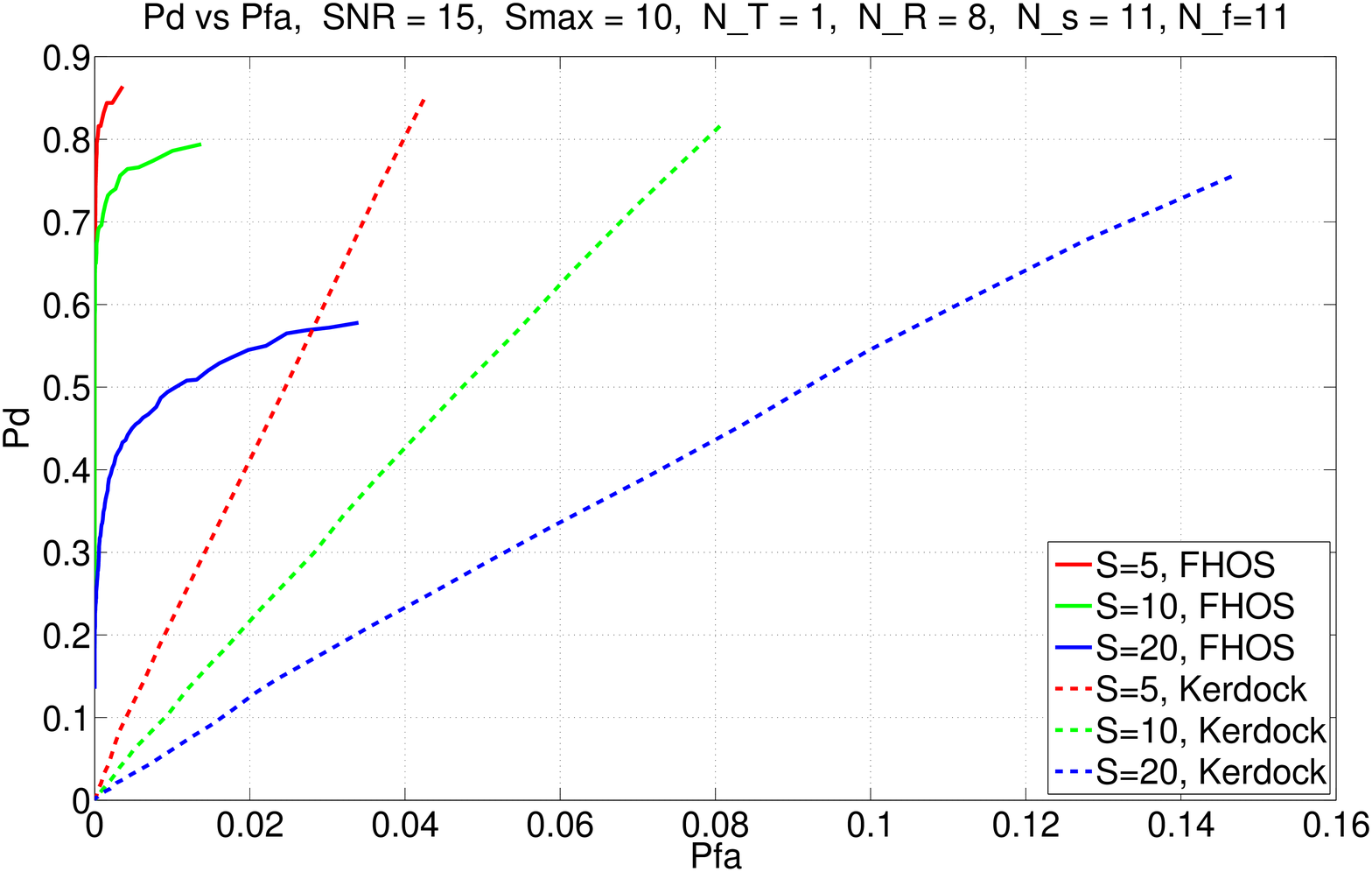} 
 
\end{tabular}
  \caption{SIMO, Kerdock codes vs FHOS, SNR=15} \label{fos1}
\end{figure}

\begin{figure}[h!]
\centering
\begin{tabular}{c}
  \includegraphics[width=80mm]{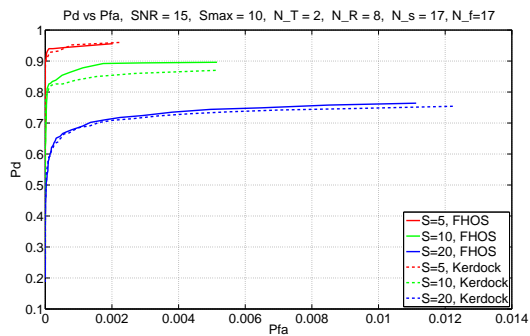} 
 \end{tabular}

 \caption{MIMO, Kerdock codes vs FHOS, SNR=15} \label{fos2}
\end{figure}


\begin{thebibliography}{10}

\bibitem{Alltop}
W.~O. Alltop.
\newblock Complex sequences with low periodic correlations.
\newblock {\em IEEE Trans. on Information Theory}, 26(3):350--354, 1980.

\bibitem{BCG11}
S.~Becker, E.~Candes, and M.~Grant.
\newblock Templates for convex cone problems with applications to sparse signal
  recovery.
\newblock {\em Mathematical Programming Computation}, 3(3):165--218, 2011.

\bibitem{BBW12}
J.J. Benedetto, R.~Benedetto, and J.T. Woodworth.
\newblock Optimal ambiguity functions and weil's exponential sum bound.
\newblock {\em Journal of Fourier Analysis and Applications}, 18:471--487,
  2012.

\bibitem{CCK97}
A.~R. Calderbank, P.~J. Cameron, W.~M. Kantor, and J.~J. Seidel.
\newblock ${Z}\sb 4$-{K}erdock codes, orthogonal spreads, and extremal
  {E}uclidean line-sets.
\newblock {\em Proc. London Math. Soc. (3)}, 75(2):436--480, 1997.

\bibitem{CP08}
E.J. Cand\`es and Y.~Plan.
\newblock {Near-ideal model selection by $\ell_1$ minimization}.
\newblock {\em Annals of Statistics}, 37(5A):2145--2177, 2009.

\bibitem{Car09}
L.~Carin.
\newblock On the relationship between compressive sensing and random sensor
  arrays.
\newblock {\em IEEE Antennas and Propagation Magazine}, 51(5):72--81, 2009.

\bibitem{CBBGS09}
V.~Cevher, P.~T. Boufounos, R.~G. Baraniuk, A.~C. Gilbert, and M.~J. Strauss.
\newblock Near-optimal bayesian localization via incoherence and sparsity.
\newblock In {\em Proc. of the 2009 International Conference on Information
  Processing in Sensor Networks (IPSN)}, pages 205--216, April 13-16 2009.

\bibitem{CSP11}
Y.~Chi, L.L. Scharf, A.~Pezeshki, and A.R. Calderbank.
\newblock Sensitivity to basis mismatch in compressed sensing.
\newblock {\em IEEE Trans. Signal Processing}, 59(5):2182--2195, 2011.

\bibitem{Cos84}
J.P. Costas.
\newblock A study of a class of detection waveforms having nearly ideal
  range-doppler ambiguity properties.
\newblock {\em Proceedings of the IEEE}, 72(8):996--1009, 1984.

\bibitem{FW12}
A.~Fannjiang and W.~Liao.
\newblock Coherence pattern–-guided compressive sensing with unresolved
  grids.
\newblock {\em SIAM J. Imaging Sci.}, 5:179--202, 2012.

\bibitem{FTD00}
A.J. Fenn, D.H. Temme, W.P. Delaney, and W.E. Courtney.
\newblock The development of phased-array radar technology.
\newblock {\em Lincoln Lab. J.}, 12(2):321--340, 2000.

\bibitem{mimobook}
B.~Friedlander.
\newblock Adaptive {S}ignal {D}esign for {MIMO} {R}adar.
\newblock In J.~Li and P.~Stoica, editors, {\em MIMO {R}adar {S}ignal
  {P}rocessing}, chapter~5. John Wiley \& Sons, 2009.

\bibitem{GG05}
S.W. Golomb and G.~Gong.
\newblock {\em Signal Design for Good Correlation for Wireless Communication,
  Cryptography and Radar}.
\newblock Cambridge University Press, 2005.

\bibitem{GHS08}
S.~Gurevich, R.~Hadani, and N.~Sochen.
\newblock The finite harmonic oscillator and its applications to sequences,
  communication and radar.
\newblock {\em IEEE Trans. Inf. Theory}, 54(9):4239--4253, 2008.

\bibitem{HaimovichBlumCimini2008}
A.M. Haimovich, R.S. Blum, and L.J. Cimini.
\newblock {MIMO} radar with widely separated antennas.
\newblock {\em IEEE Signal Processing Magazine}, pages 116--129, 2008.

\bibitem{HSP06}
R.~Heath, T.~Strohmer, and A.~Paulraj.
\newblock On quasi-orthogonal signatures for {CDMA} systems.
\newblock {\em IEEE Trans.\ Info.\ Theory}, 52(3):1217--1226, 2006.

\bibitem{HS09}
M.~Herman and T.~Strohmer.
\newblock High-resolution radar via compressed sensing.
\newblock {\em IEEE Trans. on Signal Processing}, 57(6):2275--2284, 2009.

\bibitem{HS10}
M.~Herman and T.~Strohmer.
\newblock General deviants: an analysis of perturbations in compressed sensing.
\newblock {\em IEEE Journal of Selected Topics in Signal Processing: Special
  Issue on Compressive Sensing}, 4(2):342--349, 2010.

\bibitem{HJ90}
R.A. Horn and C.R. Johnson.
\newblock {\em Matrix analysis}.
\newblock Cambridge University Press, Cambridge, 1990.
\newblock Corrected reprint of the 1985 original.

\bibitem{HCM06}
S.~D. Howard, A.~R. Calderbank, and W.~Moran.
\newblock The finite {Heisenberg-Weyl} groups in radar and communications.
\newblock {\em EURASIP J. Appl. Signal Process.}, 2006:1--12, 2006.

\bibitem{HRS12}
M.~{H\"{u}gel}, H.~Rauhut, and T.~Strohmer.
\newblock Remote sensing via $\ell_1$-minimization.
\newblock 2012.
\newblock Preprint.

\bibitem{IH08}
T.~Inoue and R.~W. Heath.
\newblock Kerdock codes for limited feedback mimo systems.
\newblock {\em Proc. of IEEE Int. Conf. on Acoustics, Speech and Signal
  Processing}, pages 3113--3116, 2008.

\bibitem{K72}
A.~Kerdock.
\newblock Studies of low-rate binary codes (ph.d. thesis abstr.).
\newblock {\em IEEE Trans. Inf. Theory}, 18(2):316--316, 1972.

\bibitem{Koe95}
H.~{K\"{o}nig}.
\newblock Isometric embeddings of euclidean spaces into finite-dimensional
  $\ell_p$-spaces.
\newblock {\em Banach Center Publications}, 34:79--87, 1995.

\bibitem{Lev82}
V.I. Levenstein.
\newblock Bounds on the maximal cardinality of a code with bounded modulus of
  the inner product.
\newblock {\em Soviet Math. Dokl.}, 25:526--531, 1982.

\bibitem{LiStoica2007}
J.~Li and P.~Stoica.
\newblock {MIMO Radar with Colocated Antennas: Review of Some Recent Work}.
\newblock {\em IEEE Signal Processing Magazine}, 24(5):106--114, 2007.

\bibitem{L64b}
Y.~Lo.
\newblock A mathematical theory of antenna arrays with randomly spaced element.
\newblock {\em IEEE Trans. Antennas and Propagation}, 12(3):257--268, 1964.

\bibitem{L64a}
Y.~Lo.
\newblock A probalistic approach to the problem of large antenna arrays.
\newblock {\em Journal of Research of the National Bureau of Standards},
  68D(5):1011--1019, 1964.

\bibitem{PCM08}
A.~Pezeshki, A.R. Calderbank, W.~Moran, and S.D. Howard.
\newblock Doppler resilient {G}olay complementary waveforms.
\newblock {\em IEEE Transactions on Information Theory}, 54(9):4254--4266,
  2008.

\bibitem{PRT08}
G.~E. Pfander, H.~Rauhut, and J.~Tanner.
\newblock Identification of matrices having a sparse representation.
\newblock {\em IEEE Trans. Signal Processing}, 56(11):5376--5388, 2008.

\bibitem{PEPC10}
L.~C. Potter, E.~Ertin, J.~T. Parker, and M.~Cetin.
\newblock Sparsity and compressed sensing in radar imaging.
\newblock {\em Proceedings of the IEEE}, 98(6):1006--1020, 2010.

\bibitem{Rihk}
A.~W. Rihaczek.
\newblock {\em High-Resolution Radar}.
\newblock Artech House, Boston, 1996.
\newblock (originally published: McGraw-Hill, NY, 1969).

\bibitem{SF12}
T.~Strohmer and B.~Friedlander.
\newblock Analysis of sparse {MIMO} radar.
\newblock 2012.
\newblock Preprint.

\bibitem{SH03}
T.~Strohmer and R.~Heath.
\newblock Grassmannian frames with applications to coding and communications.
\newblock {\em Applied and Computational Harmonic Analysis}, 14(3):257--275,
  2003.

\bibitem{TBS12}
G.~Tang, B.N. Bhaskar, P.~Shah, and B.~Recht.
\newblock Compressed sensing off the grid.
\newblock {\em Preprint, [arvix:1207.6053]}, 2012.

\bibitem{T96}
R.~Tibshirani.
\newblock Regression shrinkage and selection via the lasso.
\newblock {\em J. Roy. Statist. Soc. Ser. B}, 58(1):267--288, 1996.

\bibitem{WG11}
Z.~Wang and G.~Gong.
\newblock New sequences design from {W}eil representation with low
  two-dimensional correlation in both time and phase shifts.
\newblock {\em IEEE Trans. Inf. Theory}, 57(7):4600--4611, 2011.

\bibitem{Wel60}
G.~Welti.
\newblock Quaternary codes for pulsed radar.
\newblock {\em IEEE Trans.\ Inform.\ Theory}, 6(3):400--408, 1960.

\bibitem{WF89}
W.K. Wootters and B.D. Fields.
\newblock Optimal state-determination by mutually unbiased measurements.
\newblock {\em Annals of Physics}, 191(2):363--381, 1989.

\end{thebibliography}

\end{document}